\documentclass[reqno,12pt]{amsart}
\usepackage{amsmath, amssymb, amsthm,amsfonts}
\usepackage[english]{babel}
\usepackage{calligra}
\usepackage{graphicx}
\usepackage{url}
\usepackage{epstopdf}
\usepackage[a4paper,bindingoffset=0.5cm,left=2cm,right=2cm,top=2.5cm,bottom=2cm,footskip=.8cm]{geometry}

\usepackage{tikz}
\usetikzlibrary{arrows, automata,positioning,calc,shapes,decorations.pathreplacing,decorations.markings,shapes.misc,petri,topaths}
\newlength\figureheight
  \newlength\figurewidth
\usetikzlibrary{math,shapes.geometric}

\usepackage{rotating}
\usepackage{amsbsy,enumerate}
\usepackage{graphicx}
\usepackage{ccaption}
\usepackage{comment}
\usepackage{mathrsfs}
\usepackage{diagbox}

\newcommand{\vb}{\vspace{3.2mm}}

\newcommand{\R}{\mathbb{R}}

\newcommand{\p}{\mathbb{P}}

\renewcommand{\a}{\alpha}

\newcommand{\e}{\mathbb{E}}

\newcommand{\Var}{\mathbb{V}\textnormal{\textrm{ar}}}

\newcommand{\kk}[1]{{\color{cyan} #1}}

\allowdisplaybreaks

\newtheorem{lemma}{Lemma}
\newtheorem{corollary}{Corollary}

\newtheorem{theorem}{Theorem}

\newtheorem{proposition}{Proposition}

\makeatletter
\renewcommand{\fnum@figure}[1]{\textbf{\figurename~\thefigure}. }
\renewcommand{\fnum@table}[1]{\textbf{\tablename~\thetable}. }
\makeatother

\begin{document}

\title[Bounds for expected supremum of fBm with drift]{Bounds for expected supremum of\\ fractional Brownian motion with drift}

\author{Krzysztof Bisewski, Krzysztof D\c{e}bicki, Michel Mandjes}

\begin{abstract} We provide upper and lower bounds for the mean ${\mathscr M}(H)$ of
$\sup_{t\geqslant 0} \{B_H(t) - t\}$, with $B_H(\cdot)$ a zero-mean, variance-normalized version of
fractional Brownian motion with Hurst parameter $H\in(0,1)$. We find bounds in (semi-)closed-form,
distinguishing between $H\in(0,\frac{1}{2}]$ and $H\in[\frac{1}{2},1)$,
where in the former regime a numerical procedure is presented that drastically reduces the upper bound.
For $H\in(0,\frac{1}{2}]$, the ratio between the upper and lower bound is bounded, whereas for
$H\in[\frac{1}{2},1)$ the derived upper and lower bound have a strongly similar shape.
We also derive a new upper bound for the mean of $\sup_{t\in[0,1]} B_H(t)$, $H\in(0,\tfrac{1}{2}]$, which is tight around $H=\tfrac{1}{2}$.

\vb


\vb


\vb

\noindent
{\sc Affiliations.} {\it K.\ Bisewski} is with Universit\'e de Lausanne,
Quartier UNIL-Chamberonne,
B\^{a}timent Extranef,
1015 Lausanne, Switzerland;
his research is funded by the Swiss National Science Foundation Grant 200021-175752/1.
Part of this work was done while KB was visiting
Mathematical Institute, Wroc{\l}aw University,
pl.\ Grunwaldzki 2/4,
50-384 Wroc{\l}aw, Poland.

\noindent {\it K.\ D\c{e}bicki} is with Mathematical Institute, Wroc{\l}aw University,
pl.\ Grunwaldzki 2/4,
50-384 Wroc{\l}aw, Poland;
his research is partly funded by NCN Grant No 2018/31/B/ST1/00370 (2019-2022).

\noindent
{\it M.\ Mandjes} is with the Korteweg-de Vries Institute for Mathematics, University of Amsterdam,
Science Park 904, 1098 XH Amsterdam, the Netherlands. MM is also with E{\sc urandom},
Eindhoven University of Technology, Eindhoven, the Netherlands, and Amsterdam Business School,
Faculty of Economics and Business, University of Amsterdam, Amsterdam, the Netherlands;
his research is partly funded by the NWO Gravitation project N{\sc etworks}, grant number 024.002.003.

\end{abstract}

\maketitle

\section{Introduction}
Due to its capability of modelling a wide variety of correlation structures, fractional Brownian motion (fBm) is a frequently used Gaussian process. Indeed, whereas for classical Brownian motion the increments are independent, depending on the value of the Hurst parameter $H\in(0,1)$, fBm covers the cases of both negatively ($H<\frac{1}{2}$) and positively ($H>\frac{1}{2}$) correlated increments.
Owing to its broad applicability, fBm has become an
intensively studied object across a broad range of scientific disciplines, such as  physics \cite{REGN, SAGI}, biology \cite{CASP}, hydrology \cite{MONT}, mathematical finance \cite{BAIL,CONT}, insurance and risk \cite[Ch. VIII]{AA}, and operations research \cite{TAQQ}.

This paper considers the all-time supremum attained by an fBm with negative drift. This supremum clearly is a key quantity in the application areas mentioned --- think e.g.\ of ruin probabilities in the insurance context. Importantly, also in queueing theory such suprema are of great importance, due to the fact that the stationary workload has the same distribution as the supremum of the (time-reversed version of) the queue's net input process \cite[Thm.\ 5.1.1]{MBOOK}.
The main objective of our work is to analyze the {\it expected value} of the supremum attained by fBm with negative drift as a function of the Hurst parameter $H$.
As exact analysis has been beyond reach so far (apart from the Brownian case of $H=\frac{1}{2}$), we focus on identifying upper and lower bounds on its expected value.

Throughout this paper
$B_H(\cdot)$ denotes a zero-mean, variance-normalized version of  fBm with  $H\in(0,1)$. More specifically,
$B_H(\cdot)$ is a Gaussian process with stationary increments such that ${\mathbb E}\,B_H(t)=0$ for all $t\in{\mathbb R}$, and
\[\Var(B_H(t)-B_H(s)) = |t-s|^{2H}\]
for all $s,t\in{\mathbb R}.$
As mentioned, the primary focus of this paper is on deriving upper and lower bounds on the mean of the all-time supremum of fBm with negative drift. In other words, we wish to analyze, for some $c>0$,
\begin{equation}
{\mathscr M}(H,c): = {\mathbb E}\left[\sup_{t\geq0} \{B_H(t) - ct\}\right].
\end{equation}
Interestingly, exploiting fBm's self-similarity, for any $c>0$ we can express ${\mathscr M}(H,c)$ directly in terms of
${\mathscr M}(H):={\mathscr M}(H,1)$. This can be seen as follows. Renormalizing time yields, with $\gamma:=(2H-2)^{-1}$,
\begin{equation}\sup_{t\geq0} \{B_H(t) - ct\}= \sup_{t\geqslant  0} \{B_H({c^{2\gamma }t})-c\cdot c^{2\gamma }t\}
= \sup_{t\geqslant  0} \{B_H({c^{2\gamma }t})-c^{2H\gamma}t\}.\label{EXP}\end{equation}
As a consequence of the self-similarity of fBm, $B_H(x t)$ is distributed as $x^{H} B_H(t)$, and therefore $B_H({c^{2\gamma }t})$ is distributed as $c^{2H\gamma}B_H(t)$. Hence,  the random variable (\ref{EXP}) is, in the distributional sense, equal to
\begin{equation*} \sup_{t\geqslant  0} \{B_H({c^{2\gamma }t})-c^{2H\gamma}t\}= c^{2H\gamma}
\sup_{t\geqslant  0}\{B_H(t) - t\}.\end{equation*}
In other words, in order to analyze the behavior of ${\mathscr M}(H,c)$ for any $c>0$,
it suffices to consider its unit-drift counterpart ${\mathscr M}(H)$. Only in the Brownian case the value of this function
is known: for $H=\frac{1}{2}$, $\sup_{t\geqslant  0}\{B_{{1}/{2}}(t) - t\}$ is exponentially distributed with parameter~$2$,
so that ${\mathscr M}(\frac{1}{2})=\frac{1}{2}$.

A substantial body of literature  has focused on characterizing the distribution of the supremum of fBm,
either over a finite time interval (often assuming that the drift equals 0), or over an infinite time interval
(in which a negative drift ensures a finite supremum). In general terms, one could say that the vast majority of
the results obtained is of an asymptotic nature. For instance, in \cite{HUS,MAS,NAR}
a function $f(u)$ is found such that, as $u\to\infty$,
\begin{eqnarray}
f(u)\cdot {\mathbb P}\left(\sup_{t\geqslant  0}\{B_H(t) - t\}>u\right)\to 1;\label{as.1}
\end{eqnarray}
see for the precise statement \cite[Prop.\ 5.6.2]{MBOOK}.
We also refer to \cite{Deb99,Die04,HuP02} for extensions of (\ref{as.1}) to a broader class of
Gaussian processes with stationary increments and to
\cite{DOC} for a seminal paper on the corresponding logarithmic asymptotics.
Similar large-deviations results in another asymptotic regime,
namely a setting in which the Gaussian process is interpreted as the superposition of many i.i.d.\
Gaussian processes, can be found in \cite{DM}. Other asymptotic results relate to higher dimensional systems,
such as tandem queues; see e.g. \cite{KOS,MU}.
The logarithmic asymptotics of long busy periods  in fBm-driven queues have been identified in \cite{MMNU},
where it is noted that a similar approach could be relied upon so as characterize the speed of convergence of
fractional Brownian storage to its stationary limit \cite{MNG}.

While computing bounds pertaining to the extreme values attained by Gaussian processes is a large and mature
research area (see e.g.\ \cite{ADL,PIT,TAL}),
there is only a limited number of results that provide computable upper and lower bounds on the expected supremum.
A notable exception concerns the recent results by Borovkov et al.\ \cite{borovkov2016new,borovkov2017bounds}, presenting (non-asymptotic) bounds
on the expected supremum of a driftless fBm over a finite time interval (as functions of the
Hurst parameter $H$, that is). The same setting is considered in \cite{MAMA}, but a more pragmatic approach has been followed:  the objective is to  accurately fit a curve to estimated values of the expected supremum. In addition, we would like to stress that an intrinsic drawback of asymptotics is that, in the absence of error bounds, one does not know whether such asymptotic results provide any accurate approximations for instances in a pre-limit setting.
The above considerations motivate the objective of our work: identifying computable bounds on the expected supremum of fBm with drift. We remark that, the identification of such bounds is clearly relevant in its own right, but they in addition  play a pivotal role
when one aims at applying Borell-type inequalities \cite{ADL} so as to obtain uniform estimates for the tail distribution
of suprema.

\begin{figure}[b]
        \centering
       \includegraphics[height=0.45\linewidth,width=0.675\linewidth]{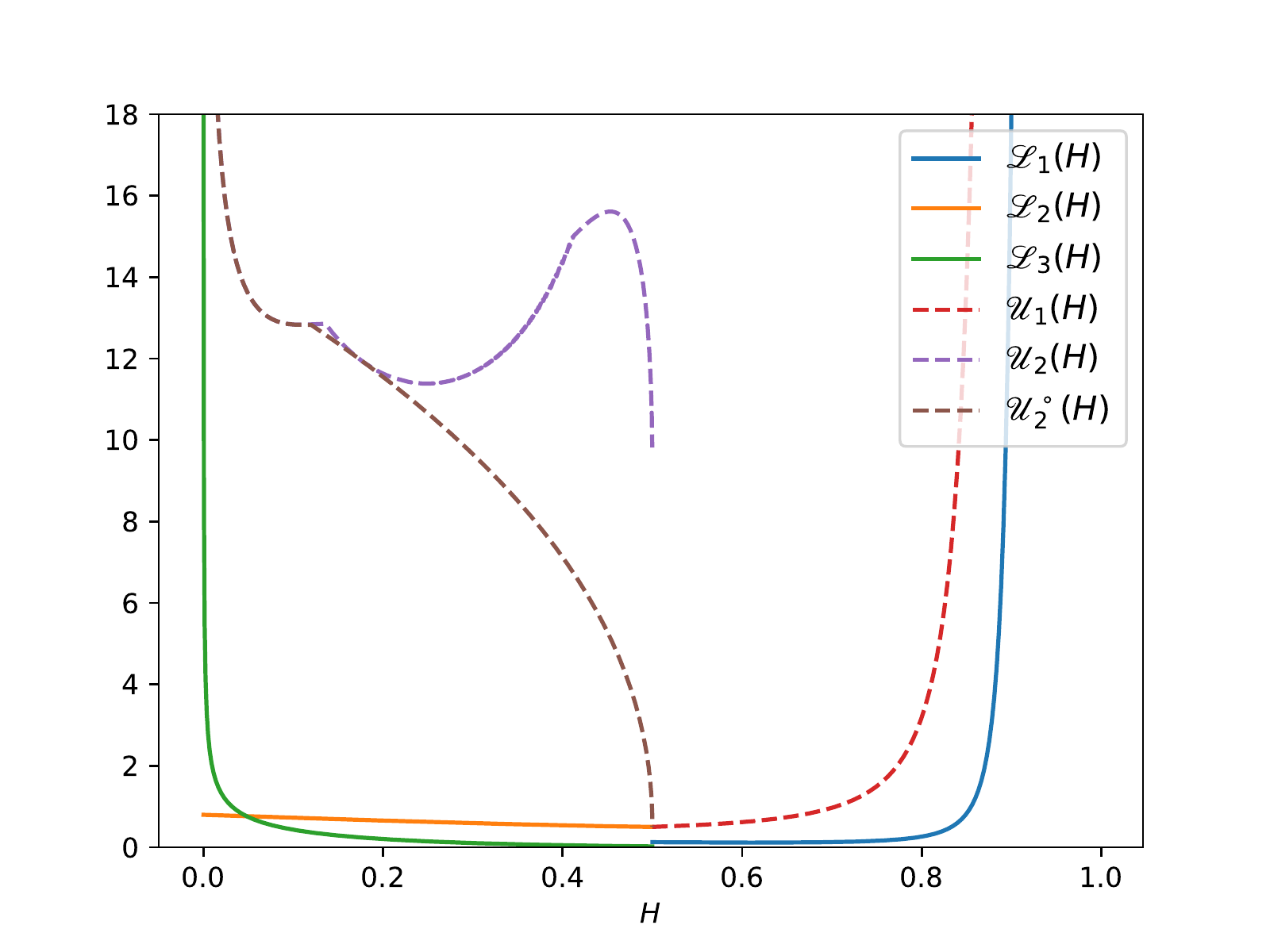}
      \caption{All upper and lower bounds for $\mathscr M(H)$ derived in this work.
      }\label{fig:all_bounds}
\end{figure}

We proceed by stating some of our results. With ${\mathscr N}$ denoting a standard normal random variable, we define, for $H\in(0,1)$,
\[{
\kappa(H):= {\mathbb E}\left[\,| {\mathscr N} |^{\frac{1}{1-H}}\,\right],
}\]
which can be given explicitly in terms of the Gamma function (see Lemma \ref{KAPPA}).
The first main result concerns the behavior of ${\mathscr M}(H)$ for $H\downarrow 0$ and $H\uparrow 1$ respectively.
\begin{theorem}\label{thm:infsup} We have
\begin{equation}\label{eq:infsup0}
0.2055 \approx \frac{1}{2\sqrt{\pi e \log 2}} \leqslant \liminf_{H\downarrow 0} \frac{{\mathscr M}(H)}{\sqrt{H}} \leqslant \limsup_{H\downarrow 0}  \frac{{\mathscr M}(H)}{\sqrt{H}} \leqslant  1.695,
\end{equation}
and
\begin{equation}\label{eq:infsup1}
\textcolor{black}{ \liminf_{H\uparrow 1} \frac{{\mathscr M}(H)}{(1-H)\,\kappa({H})}\geqslant
\frac{1}{2e}},\:\:\:\:\:\:
 \limsup_{H\uparrow 1}  \frac{{\mathscr M}(H)}{\kappa({H})}\leqslant
 \frac{1}{2}.
\end{equation}
\end{theorem}
The asymptotic inequalities (\ref{eq:infsup1}), in combination with the exact value of $\kappa({H})$
derived in Lemma \ref{lem:kap} in Section \ref{sec:SUP}, straightforwardly imply that ${\mathscr M}(H)$ and $\kappa(H)$ `logarithmically match' as $H\uparrow 1$, in the sense that
\[
\lim_{H\uparrow 1}\frac{\log {\mathscr M}(H)}{\log \kappa({H})}=1.
\]
The second main result concerns bounds for $H\in(0,1)$. These differ by at most a multiplicative constant that is uniformly bounded over $H\in(0,\frac{1}{2}]$, whereas they differ by at most a factor $e/(1-H)$ for $H\in[\frac{1}{2},1).$  We define
\begin{align*}
\mathscr U(H) :=
\begin{cases}
\mathscr U_2^\circ(H), & H\in(0,\tfrac{1}{2}] \\
\mathscr U_1(H), & H\in[\tfrac{1}{2},1)
\end{cases}
\quad\quad
\mathscr L(H) :=
\begin{cases}
\max\{\mathscr L_2(H), \mathscr L_3(H)\}, & H\in(0,\tfrac{1}{2}] \\
\mathscr L_1(H), & H\in[\tfrac{1}{2},1)
\end{cases}
\end{align*}
with functions $\mathscr U_1(\cdot), \mathscr L_1(\cdot), \mathscr L_2(\cdot), \mathscr L_3(\cdot)$ that are defined in Propositions \ref{thm:AMI}--\ref{thm:LB_CJB}, and a function $\mathscr U_2^\circ(\cdot)$ that is defined in Corollary \ref{cor:sudakov_Hlow} and that uses the function $\mathscr U_2(\cdot)$ from Proposition
\ref{thm:UB_Cinv}.
As such, $\mathscr U(H)$ and $\mathscr L(H)$ are the best upper and lower bound for $\mathscr M(H)$ that we were able to find. It is noted that all these functions can be computed through elementary numerical procedures.

\begin{theorem}\label{thm:slon}
We have  ${\mathscr L}(H)$ and ${\mathscr U}(H)$ that satisfy
\[{\mathscr L}(H)\leqslant {\mathscr M}(H) \leqslant {\mathscr U}(H);\]
such that
\[\sup_{H\in(0,\frac{1}{2}]}\frac{{\mathscr U}(H)}{{\mathscr L}(H)} \leqslant \textcolor{black}{18.063},\]
whereas for $H\in[\frac{1}{2},1)$ it holds that $2/(1-H) \leqslant {\mathscr U}(H)/{\mathscr L}(H) \leqslant e/(1-H).$
\end{theorem}
The proofs of Theorems \ref{thm:infsup} and Theorem \ref{thm:slon} will be given later in the paper, as well as the procedure to compute
${\mathscr L}(H)$ and ${\mathscr U}(H)$ in
Theorem \ref{thm:slon}. The bounds $\mathscr U_1(\cdot), \mathscr L_1(\cdot), \mathscr L_2(\cdot), \mathscr L_3(\cdot)$ are explicit  functions of $H$, whereas the tightest upper bound $\mathscr U_2^\circ(\cdot)$ follows by performing a numerical procedure on the (semi-)closed-form upper bound $\mathscr U_2(\cdot)$.
The resulting bounds are summarized in Figure~\ref{fig:all_bounds}. We note that the bounds are tight in $H=\frac{1}{2}.$
Notably, as a by-product of the proof for the upper bound ${\mathscr U_2}^\circ(\cdot)$, we present in Corollary \ref{cor:sudakov_Hlow_supXt} for the regime that $H\in(0,\frac{1}{2}]$ a new upper bound on  the mean of $\sup_{t\in[0,1]} B_H(t)$. Also this bound
is tight at $H=\frac{1}{2}$,
and improves the upper bound that was established in
\cite{borovkov2016new}.

This paper is organized as follows. As it turns out, we have to separately consider
the cases $H\in(0,\frac{1}{2}]$ and $H\in[\frac{1}{2},1)$.
As in the former case   $\Var\, B_H(t)$ grows slower than linearly, in the physics literature  \cite{MERO,METZ} this regime is sometimes referred to as {\it subdiffusive}. Analogously, in the latter case $\Var\, B_H(t)$ is superlinear in $t$, explaining why this regime is called {\it superdiffusive}.
Section \ref{sec:SUP}, dealing with the superdiffusive case, presents an upper and lower bound that have a strongly similar shape. Then, in Section \ref{sec:SUB}, we focus on the subdiffusive regime, with an upper bound and two lower bounds (one of them being tighter for small $H\in(0,\frac{1}{2}]$, the other one being tighter for larger $H$). This section also presents additional bounds and a procedure to numerically improve the upper bound. Section \ref{thm:proofs} covers the proofs of Theorems~\ref{thm:infsup}  and \ref{thm:slon}. The paper is concluded in Section \ref{CONC}.

\section{Superdiffusive regime}\label{sec:SUP}
In this section we consider the case that $H\in[\frac{1}{2},1)$.  We start our exposition with a useful auxiliary result; see also \cite{WIN}.

\begin{lemma} \label{lem:kap} For $H\in(0,1)$,
\[\kappa(H) =\sqrt{\frac{1}{\pi}}
(\sqrt{2})^{\frac{1}{1-H}} \Gamma\left(\frac{2-H}{2-2H}\right).\]
\label{KAPPA}
\end{lemma}
\begin{proof}
Observe that, performing the change-of-variable $x^2=2y$,
\[\kappa(H) = \sqrt{\frac{2}{\pi}} \int_0^\infty e^{-x^2/2} x^{\frac{1}{1-H}}{\rm d}x=\sqrt{\frac{1}{\pi}}
(\sqrt{2})^{\frac{1}{1-H}} \int_0^\infty e^{-y} y^{\frac{H}{2(1-H)}}{\rm d}y,\]
which, after interpreting the integral in terms of the Gamma function, proves the claim.
\end{proof}

With this lemma at our disposal, we can present our lower bound. It relies on the `principle of the largest term': the probability of a union of events is bounded from below by the probability of the most likely among these events.
In the sequel, the following quantity features regularly:
\[
\nu(H):=H^H(1-H)^{1-H}.\]

\begin{proposition}\label{thm:AMI}
For $H\in[\frac{1}{2},1)$ we have ${\mathscr M}(H)\geqslant {\mathscr L}_1(H)$, where
\begin{equation}\label{eq:LB_AMI}
{\mathscr L}_1(H) := \tfrac{1}{2}\nu(H)^{\tfrac{1}{1-H}} \cdot \kappa(H).
\end{equation}
\end{proposition}
As noted from the proof, the lower bound ${\mathscr L}_1(H)$ holds in the entire domain, i.e., for any $H\in(0,1)$.
\begin{proof} This proof is due to \cite{IN}.
Let $t_u$ be the maximizer, for a given $u>0$, of the mapping \[t\mapsto\p(B_H(t)-t > u)=
1-\Phi\left(\frac{u+t}{t^{H}}\right),\] with $\Phi(\cdot)$ the cumulative distribution function of a standard normal random variable.  As $\Phi(\cdot)$ is increasing $t_u$ is the minimizer of $(u+t)/t^{H}$, i.e.,
\[t_u = u\,\frac{H}{1-H}.\]
This means that we have the lower bound
\begin{align*}
{\mathscr M}(H)&=\int_0^\infty \p\left(\sup_{t\geqslant 0}\{B_H(t) - t\} > u\right){\rm d}u \geqslant \int_0^\infty \p(B_H(t_u) - t_u > u)\,{\rm d}u\\&=\int_0^\infty \left(1-\Phi\left(\frac{u+t_u}{t_u^{H}}\right)\right){\rm d}u=\int_0^\infty \left(1-\Phi\left(\frac{u^{1-H}}{\nu(H)}\right)\right){\rm d}u\\
&=H^{\frac{H}{1-H}}\int_0^\infty v^{\frac{H}{1-H}}\left(1-\Phi(v)\right){\rm d}v.
\end{align*}
By an elementary application of integration by parts, we obtain that the integral in the last expression can be written as
\begin{align*}\int_0^\infty (1-\Phi(v)) \,{\rm d}\left((1-H) v^{\frac{1}{1-H}}\right) =
(1-H) \int_0^\infty v^{\frac{1}{1-H}} \Phi'(v)\,{\rm d}v=
\tfrac{1}{2}(1-H)\,\kappa(H),\end{align*}
from which \eqref{eq:LB_AMI} follows.
\end{proof}

We proceed by deriving an upper bound. It will make use of the following results that have been proven in D\c{e}bicki et al.\
\cite{D1,D2}. We define
\[\lambda(u,H):=\left(\int_0^\infty (2\pi t^{2H})^{-1/2}\exp(-(t+u)^2/(2t^{2H})){\rm d}t\right)^{-1}.\]

\begin{lemma}\label{thm:krzysD}
Denote by $B(\cdot)\equiv B_{\frac{1}{2}}(\cdot)$ a standard Brownian motion. Then, for any $H\in(0,1)$ we have
\begin{equation}\label{eq:KrzysD}
2-2H \leqslant \lambda(u,H)\cdot\p\left(\sup_{t\geqslant 0} \{B (t^{2H}) - t\} > u\right) \leqslant 2.
\end{equation}
For $H\geqslant \frac{1}{2}$, a tighter upper bound holds:
\begin{equation}\label{eq:DMR_ineq}
 \lambda(u,H)\cdot\p\left(\sup_{t\geqslant 0} \{B(t^{2H}) - t \}> u\right) \leqslant 1.
\end{equation}
\end{lemma}
\begin{proof}
See \cite{D1} for the proof of \eqref{eq:KrzysD} and \cite{D2} for the proof of \eqref{eq:DMR_ineq}.
\end{proof}

\begin{proposition}\label{thm:UB_K2001}
For $H\in[\frac{1}{2},1)$ we have ${\mathscr M}(H)\leqslant {\mathscr U}_1(H)$, where
\begin{equation}
{\mathscr U}_1(H)= \tfrac{1}{2}\cdot \kappa(H).
\end{equation}
\end{proposition}
\begin{proof}
Informally, Slepian's lemma compares the tail probabilities pertaining to the suprema of
two Gaussian processes in the
situation that the corresponding variance and mean functions are equal and the variograms are ordered;
see for the precise statement e.g.\ \cite{ADL}.
In the case of $H\in[ \frac{1}{2},1)$, applying Slepian's lemma yields that
\begin{eqnarray}
\p\left(\sup_{t\geqslant0} \{B_H(t) - t \}> u\right)
\leqslant
\p\left(\sup_{t\geqslant 0} \{B(t^{2H}) - t \}> u\right),\label{slep}
\end{eqnarray}
for all $u\in \R$.

Using this bound in the first inequality, and Eqn.\ \eqref{eq:DMR_ineq} from Lemma~\ref{thm:krzysD} in the second inequality, we obtain
\begin{align*}
{\mathscr M}(H)& = \int_0^\infty \p\left(\sup_{t\geqslant0} \{B_H(t) - t \}> u\right){\rm d}u
\leqslant
\int_0^\infty \p\left(\sup_{t\geqslant 0} \{B(t^{2H}) - t \}> u\right)\,{\rm d}u \\
& \leqslant \int_0^\infty\int_0^\infty \frac{1}{\sqrt{2\pi t^{2H}}}\exp\left(-\frac{(t+u)^2}{2t^{2H}}\right){\rm d}u\, {\rm d}t = \int_0^\infty \p(B(t^{2H}) > t)\,{\rm d}t \\
& = \tfrac{1}{2}\int_0^\infty \p(|\mathscr N| > t^{1-H})\,{\rm d}t= \tfrac{1}{2}\int_0^\infty \p(|\mathscr N|^{\frac{1}{1-H}} > t)\,{\rm d}t,
\end{align*}
which equals $\tfrac{1}{2}\cdot \kappa(H).$
\end{proof}

\section{Subdiffusive regime}\label{sec:SUB}
This section covers the case that $H\in(0,\frac{1}{2}]$.
Section \ref{sec:CB}  provides various bounds that allow (semi-)closed-form expressions. Then in Section \ref{sec:NT} we develop a numerical procedure that improves the upper bound, and in addition we present bounds on various quantities featuring in Section \ref{sec:CB}.

\subsection{Closed-form bounds}\label{sec:CB} We start by a theorem that is the immediate counterpart of Proposition \ref{thm:UB_K2001}. Due to the nature of Slepian's inequality, for $H\in(0,\frac{1}{2}]$ it  constitutes a lower bound, though. As its proof is essentially analogous to that of Proposition \ref{thm:UB_K2001}, we only provide the main steps.

\begin{proposition}\label{thm:LB_K2001}
For $H\in(0,\frac{1}{2})$ we have ${\mathscr M}(H)\geqslant {\mathscr L}_2(H)$, where
\begin{equation}
{\mathscr L}_2(H) := (1-H)\cdot \kappa(H).
\end{equation}
\end{proposition}
\begin{proof}
For $H\geqslant \frac{1}{2}$, Slepian's lemma yields
\[\int_0^\infty \p\left(\sup_{t\geqslant0} \{B_H(t) - t \}> u\right){\rm d}u \geqslant \int_0^\infty \p\left(\sup_{t\geqslant 0} \{B(t^{2H}) - t \}> u\right)\,{\rm d}u,\]
which majorizes, by the lower bound in \eqref{eq:KrzysD},
\[(2-2H)\int_0^\infty\int_0^\infty \frac{1}{\sqrt{2\pi t^{2H}}}\exp\left(-\frac{(t+u)^2}{2t^{2H}}\right){\rm d}u\, {\rm d}t,\]
equalling $(1-H)\cdot \kappa(H).$
\end{proof}

We now present a second lower bound, ${\mathscr L}_3(H)$.
It is noted that this ${\mathscr L}_3(H)$ provides a lower bound on ${\mathscr M}(H)$ for {\it all} $H\in(0,1)$,
but for $H\in(\tfrac{1}{2},1)$, it performs worse than the bound ${\mathscr L}_1(H)$ from Proposition~\ref{thm:LB_K2001}.
In this lower bound the object \begin{equation}
\label{defmu}\mu(H):=\e\left[\left(\sup_{t\in[0,1]} B_H(t)\right)^{\frac{1}{1-H}}\right]\end{equation}
plays a key role.
This quantity will be analyzed in greater detail in Section \ref{sec:NT}:
while we lack a closed-form expression for $\mu(H)$, we derive upper and lower bounds.
From Lemma \ref{lem:a-moment_borel} and Corollary \ref{cor:sudakov_Hlow_supXt},
which will be stated and proven in Section \ref{sec:NT}, we conclude that
\[\underline\mu(H) \leqslant \mu(H) \leqslant \overline\mu(H),\]
with
\begin{equation}\label{eq:P1P2}
\underline\mu(H) := \left(\frac{C^-}{\sqrt{ H}}\right)^{\frac{1}{1-H}}, \:\:\overline\mu(H) :=  \left(\overline\mu^\circ(H,1)\right)^{\frac{1}{1-H}} + \textcolor{black}{\frac{\sqrt{\pi/2}}{1-H}\left((\overline\mu^\circ(H,1))^{\varsigma(H)}+ \e|\mathscr N|^{\varsigma(H)}\right),}
\end{equation}
where {$\mathscr N$} denotes a standard normal  random variable, \textcolor{black}{$\varsigma(H):=H/(1-H)$, and} $C^-$ and $\overline\mu^\circ(\cdot, 1)$ are introduced in Lemma \ref{lem:a-moment_borel} and Corollary
\ref{cor:sudakov_Hlow_supXt}. \textcolor{black}{It is noted that $\e|\mathscr N|^{\varsigma(H)}$ can be expressed in terms of the Gamma function, similarly to how this was done in Lemma~\ref{lem:kap}.}

\begin{proposition}\label{thm:LB_CJB}
For $H\in(0,\frac{1}{2}]$ we have ${\mathscr M}(H)\geqslant {\mathscr L}_3(H)$, where
\begin{equation}\label{eq:LB_CJB}
 {\mathscr L}_3(H) := \nu(H)^{\frac{1}{1-H}} \cdot \underline\mu(H).
\end{equation}
\end{proposition}

\begin{proof}
Using the time-scaling property of fBm we obtain that ${\mathscr M}(H)$ equals
\[  \int_0^\infty \p\left(\sup_{t\geqslant 0}\{B_H(t) - t\} > u\right){\rm d}u = \int_0^\infty \p\left(\sup_{t\geqslant0}\frac{B_H(t)}{1+t} > u^{1-H}\right){\rm d}u  = \e \left[\left(\sup_{t\geqslant 0}\frac{B_H(t)}{1+t}\right)^{\frac{1}{1-H}}\right].\]
The last expression in the previous display obviously majorizes, for any $T>0$,
\[\e\left[ \left(\sup_{t\in[0,T]}\frac{B_H(t)}{1+t}\right)^{\frac{1}{1-H}}\right].\]
Again using the time-scaling property of fBm, we see that
\begin{align*}
\e\left[\left( \sup_{t\in[0,T]}\frac{B_H(t)}{1+t}\right)^{\frac{1}{1-H}}\right] \geqslant \left(\frac{T^H}{1+T}\right)^{\frac{1}{1-H}} \cdot\e\left[\left( \sup_{t\in[0,1]}B_H(t)\right)^{\frac{1}{1-H}}\right].
\end{align*}
For a given $H$, the maximum of a function $T\mapsto T^H/(1+T)$ is attained at $T = H/(1-H)$ and equals $\nu(H)$, which in combination with \eqref{eq:P1P2} yields the bound in \eqref{eq:LB_CJB}.
\end{proof}

Numerical experiments show that  ${\mathscr L}_3(H)$ is the tightest of the two lower bounds for $H$ close to 0, whereas ${\mathscr L}_2(H)$ is the tightest for larger values of $H$.

The next objective is to derive an upper bound. In this result, we extensively use the quantity
\[\psi(T,H):=\min\left\{T\cdot \frac{1-2H}{2H}, 1\right\},\]
which is non-negative for $H\in(0,\frac{1}{2}]$. After the proof of the following proposition, we will comment on the computation of the quantity $\omega(H)$ that features in this upper bound. We recall that $\overline\mu(H)$ was defined in \eqref{eq:P1P2}.

\begin{proposition}\label{thm:UB_Cinv}
For $H\in(0,\frac{1}{2}]$ we have ${\mathscr M}(H)\leqslant {\mathscr U}_2(H)$, where
\begin{equation}\label{eq:UB_Cinv}
{\mathscr U}_2(H) := \omega(H) \cdot \overline\mu(H),
\end{equation}
and
\begin{align}\label{eq:mH}
\omega(H) :=  \inf_{T>0} \left\{T^{\frac{H}{1-H}} + \left(\frac{\psi(T,H)^{1-2H}\, T^H}{\psi(T,H)+T} \right)^{\frac{1}{1-H}} \right\}.
\end{align}

In addition, $\omega(H) \leqslant \min\{\omega_1(H), \omega_2(H)\}$, where
\begin{align}\label{eq:m2m3}
\omega_1(H) := 2\cdot (2H)^{\frac{H}{1-H}}(1-2H)^{\frac{1-2H}{2-2H}}, \quad \quad \omega_2(H) := \frac{1}{\nu(H)}.
\end{align}
\end{proposition}

\begin{proof}
Our starting point is again
\[{\mathscr M}(H) = \int_0^\infty \p\left(\sup_{t\geqslant0}\frac{B_H(t)}{1+t} > u^{1-H}\right){\rm d}u.\]
The idea is now to split the interval $[0,\infty)$ into $[0,T)$ and $[T,\infty)$, in the sense that the expression in the previous display is majorized by
\begin{align*}
 \int_0^\infty &\p\left(\sup_{t\in[0,T]}\frac{B_H(t)}{1+t} > u^{1-H}\right){\rm d}u +  \int_0^\infty \p\left(\sup_{t\in[T,\infty)}\frac{B_H(t)}{1+t} > u^{1-H}\right){\rm d}u \\
& = \e \left[\left(\sup_{t\in[0,T]}\frac{B_H(t)}{1+t}\right)^{\frac{1}{1-H}}\right] + \e \left[\left(\sup_{t\in[T,\infty)}\frac{B_H(t)}{1+t}\right)^{\frac{1}{1-H}}\right].
\end{align*}
The next step is to consider these two terms separately. We deal with the second term {using self-similarity and the fact that} \[{\big\{B_H(t)\big\}_{t\in\R_+} \stackrel{\rm d}{=} \big\{t^{2H}B_H(1/t)\big\}_{t\in\R_+},}\] with the objective to arrive at an expression similar to the first term. Concentrating on this second term, we thus obtain
\begin{align*}
\sup_{t\in[T,\infty)}\frac{B_H(t)}{1+t} & = \sup_{t\in(0,1/T]}\frac{B_H(1/t)}{1+1/t} \stackrel{\rm d}{=}  \sup_{t\in(0,1/T]}\frac{t^{-2H}B_H(t)}{1+1/t} = \sup_{t\in(0,1]}\frac{(t/T)^{-2H}B_H(t/T)}{1+T/t} \\
& \stackrel{\rm d}{=} \sup_{t\in(0,1]}\frac{t^{-2H}T^H}{1+T/t} \cdot B_H(t) = \sup_{t\in(0,1]}\frac{t^{1-2H}T^H}{t+T} \cdot B_H(t).
\end{align*}
Upon combining the above,
\begin{align}
\nonumber \e&\left[ \left(\sup_{t\in[0,T]}\frac{B_H(t)}{1+t}\right)^{\frac{1}{1-H}}\right] + \e\left[ \left(\sup_{t\in[T,\infty)}\frac{B_H(t)}{1+t}\right)^{\frac{1}{1-H}}\right] \\
\nonumber& = \e \left[\left( \sup_{t\in[0,1]} \frac{T^H}{1+tT} \cdot B_H(t) \right)^{\frac{1}{1-H}}\right]+ \e \left[\left( \sup_{t\in[0,1]} \frac{t^{1-2H}T^H}{t+T} \cdot B_H(t) \right)^{\frac{1}{1-H}}\right] \\
\label{eq:to_dissect}& \leqslant \left( T^{\tfrac{H}{1-H}} + \left(\sup_{t\in[0,1]} \frac{t^{1-2H} T^H}{t+T} \right)^{\frac{1}{1-H}} \right) \cdot \e\left[ \left( \sup_{t\in[0,1]} B_H(t) \right)^{\frac{1}{1-H}}\right].
\end{align}
It takes some elementary calculus to verify that, for given values of  $T$ and $H$, the supremum of the function $t\mapsto {t^{1-2H}}/({t+T})$ over $[0,1]$ is attained at $\psi(T,H)$. This, combined with \eqref{eq:P1P2}, proves that, indeed, \eqref{eq:UB_Cinv} holds with  $\omega(H)$ as defined  in \eqref{eq:mH}.

It is left to show that $\omega(H) \leqslant \min\{\omega_1(H), \omega_2(H)\}$.
The bound $\omega(H) \leqslant \omega_1(H)$ results from taking the infimum in $\omega(H)$ in \eqref{eq:mH}, but over a subinterval of $(0,\infty)$. More concretely, we consider the
interval $(0,\tau(H))$ with $\tau(H):=2H/(1-2H)$, in which we can replace $\psi(T,H)$ by $T(1-2H)/(2H)$. We obtain
\begin{align}\nonumber
\omega(H)& \leqslant\omega_1(H) = \inf_{T \in(0,\tau(H))}\left\{T^{\frac{H}{1-H}} + \left(\frac{\psi(T,H)^{1-2H} T^H}{\psi(T,H)+T} \right)^{\frac{1}{1-H}} \right\}\\
\label{eq:infT}&= \inf_{T \in(0,\tau(H))}\left(T^{\frac{H}{1-H}}+ T^{-\frac{H}{1-H}}(1-2H)^{\frac{1-2H}{1-H}}(2H)^{\frac{2H}{1-H}}\right).
\end{align}
Computing the derivative with respect to $T$ and solving the first-order condition yields that infimum above is attained at \[T = 2H(1-2H)^{\frac{1-2H}{2H}}.\]
It is directly seen that this minimizer does not exceed  ${2H}/({1-2H})$, and hence it is also the minimizer of \eqref{eq:infT}. Further standard algebraic manipulations yield the expression in \eqref{eq:m2m3}.

Further, the bound $\omega(H) \leqslant \omega_2(H)$ results from realizing that $\psi(T,H)\leqslant 1$:
\begin{align*}
\omega(H) \leqslant \omega_2(H) = \inf_{T > 0}\left\{T^{\frac{H}{1-H}} + \left(\frac{T^H}{T} \right)^{\frac{1}{1-H}} \right\}.
\end{align*}
The infimum above is attained at \[T = \left(\frac{1-H}{H}\right)^{1-H}
\] and again, simple algebra then directly leads to the expression $1/\nu(H)$ in \eqref{eq:m2m3}.
\end{proof}

Although we have the two upper bounds from \eqref{eq:m2m3}, it is worthwhile to explore whether we can analyze $\omega(H)$ in a more precise fashion. The next lemma deals with this issue. Here $H_0 \approx 0.1541$ is the unique solution to the equation
\begin{equation}\label{eq:H0}\frac{H}{1-H} = \left(\frac{2-H}{1-H}\right)^{-\frac{2-H}{1-H}},\end{equation}
and  $\tau^\circ(H)$ is the unique solution to the equation, for $\tau\geqslant (1-H)^{-1}$,
\begin{align}\label{eq:taucirc}
\frac{H}{1-H} + \left({\frac{H}{1-H}-\tau}\right){(1+\tau)^{-\frac{2-H}{1-H}}} = 0.
\end{align}

\begin{lemma}\label{rem:UB_Cinv} The function $\omega(H)$, as defined in \eqref{eq:mH}, with $H\in(0,\frac{1}{2}]$, satisfies
\begin{align}
\omega(H) = \begin{cases}
\min\{\omega_0(H), \omega_1(H)\}, &  H\leqslant H_0 \\
\omega_1(H), &  H > H_0,
\end{cases}
\end{align}
where
\begin{align}\label{eq:m1}
\omega_0(H) = \tau^\circ(H)^{\frac{H}{1-H}} \left(1+ \frac{1}{{(1+\tau^\circ(H))^{\frac{1}{1-H}}}}\right).
\end{align}
\end{lemma}

\begin{proof}
Above we already  considered the infimum in \eqref{eq:mH}, but with the minimization performed only over $T<\tau(H):=2H/(1-2H).$ This resulted in the expression for $\omega_1(H)$ as given in \eqref{eq:m2m3}.

We continue by considering the infimum in \eqref{eq:mH}, but now with the minimization performed only over $T\geqslant \tau(H).$ To this end, we define the functions
\[F(T) := T^\a\left(1 + \frac{1}{(1+T)^{\a+1}}\right),\hspace{1cm}f(T) := \a + \frac{\a-T}{(1+T)^{\a+2}}.\]
Hence, for a given value of $H\in(0,\frac{1}{2}]$,  the infimum we are looking for is $\inf_T F(T)$, with $T>\tau(H)$ and $\alpha=H(1-H)^{-1}$. It is directly seen that $F'(T) = T^{\a-1} f(T)$, so that the
first-order condition reduces to $f(T)=0.$ We also have that $f(0) = 2\a$, $\lim_{T\to\infty}f(T) = \a$, and
\begin{align*}
f'(T) = \frac{1+\a}{(1+T)^{3+\a}} \cdot (T - (1+\a)).
\end{align*}
This  means that the function $f(\cdot)$ is strictly decreasing on $T\in[0,1+\a)$ and strictly increasing on $(1+\a,\infty)$, i.e., it attains its minimum at $T = 1+\a$. As a consequence, the function $f(\cdot)$ has {\it at most two zeros}. We distinguish between  three cases.
\begin{itemize}
\item[(1)] Suppose $f(1+\a) > 0$.  In this case the equation $f(T) = 0$ does not have any positive solutions and thus $F'(T) > 0$ for all $T$, meaning that the infimum of function $F(\cdot)$ over $T>\tau(H)$ is attained at $T=\tau(H)$.
\item[(2)] Suppose $f(1+\a) = 0$. \ Then the equation $f(T)=0$ has exactly one solution, namely $T = 1+\a = ({1-H})^{-1}$. This means that the infimum of $F(\cdot)$ over $T>\tau(H)$ is attained at $\max\{\tau(H), ({1-H})^{-1}\}$.
\item[(3)] Suppose $f(1+\a) < 0$. \ Then the equation $f(T)=0$ has at most two solutions, say $T_1(H)$ and  $T_2(H)$, where {$T_1(H) < 1+\a < T_2(H)$}. The infimum of $F(\cdot)$ over $T>\tau(H)$ is then attained at $\max\{\tau(H), T_2(H)\}$. Finally, since $\tau(H) \leqslant 1+\a$ for $H\in(0,1-\frac{1}{2}\sqrt{2}]$, we remark that the maximum is attained at $T_2(H)$ as long as $H$ belongs to that interval. Observe that $T_2(H)$ solves \eqref{eq:taucirc}, so that we can identify it with $\tau^\circ(H)$.
\end{itemize}
Now that we have analyzed the minimum over $T<\tau(H)$ and $T\geqslant \tau(H)$, we have to pick the smallest of these numbers.
Across all $H\in(0,\frac{1}{2}]$ we have that $F(\tau(H)) \geqslant \omega_1(H)$, because
\[\frac{F(\tau(H))}{\omega_1(H)} = \tfrac{1}{2} \cdot \left(\beta + \frac{1}{\beta}\right),\]
where $\beta = (1-2H)^{1/(2-2H)}$, in combination with the known equality
$\beta + \beta^{-1} \geqslant 2$, for any $\beta>0.$
It thus suffices to compare $\omega_1(H)$ with the value of function $F(\cdot)$ at $\tau^\circ(H)$.
Observing that $f(1+\a)=0$ coincides with \eqref{eq:H0}, we obtain the desired result.
\end{proof}

\subsection{Numerical techniques for improved  bounds}\label{sec:NT} We start by stating and proving an upper and lower bound on the function $\mu(\cdot)$. These are the functions $\overline\mu(\cdot)$ and $\underline\mu(\cdot)$, which were given in \eqref{eq:P1P2} and appeared in Propositions \ref{thm:LB_CJB} and \ref{thm:UB_Cinv}. Then we focus on developing numerical procedures to find a tighter upper bound on ${\mathscr M}(H).$
We do so by studying the object
\[\mu(H,\alpha):=\e\left[\left(\sup_{t\in[0,1]} B_H(t)\right)^{\alpha}\,\right],\]
where we note that $\mu(H, (1-H)^{-1})=\mu(H)$, with $\mu(\cdot)$ as defined in \eqref{defmu}.

The next lemma presents (i)~bounds on $\mu(H,\a)$ in terms of $\mu(H,1)$, and (ii)~explicit bounds on
$\mu(H,1)$. {With $\lceil x\rceil$ denoting the smallest integer larger than or equal to $x$}, we note that $2/\log_2 \lceil 2^{2/H}\rceil = H$ when $2^{2/H}$ is an integer.

\begin{lemma}\label{lem:a-moment_borel}
For any $H\in(0,1)$ and $\a>1$,
\begin{equation}\label{eq:P1}
(\mu(H,1))^{\a} \leqslant \mu(H,\a) \leqslant \textcolor{black}{(\mu(H,1))^\alpha + \max\{1,2^{\a-2}\}\a\sqrt{\frac{\pi}{2}}\left((\mu(H,1))^{\alpha-1} + \e|\mathscr N|^{\a-1}\right).}
\end{equation}
For $H \in(0,\frac{1}{2}]$,
\begin{equation}\label{eq:P2}
\frac{C^-}{\sqrt{ H}} \leqslant \mu(H,1)  \leqslant \overline\mu(H,1) := \frac{C^+}{\sqrt{2/\log_2 \lceil 2^{2/H}\rceil}},
\end{equation}
where  $C^-: = ({2\sqrt{\pi e \log 2}})^{-1} \approx 0.2055$ and $C^+: = 1.695$.
\end{lemma}
\begin{proof}
The inequalities \eqref{eq:P2} are due to work by Borovkov et al.: for the lower bound, consult~\cite[Thm.\ 1(i)]{borovkov2017bounds} and for the upper bound consult \cite[Corollary~2]{borovkov2016new}.

Hence the inequalities \eqref{eq:P1} are left to be proven. The lower bound is an immediate  consequence of Jensen's inequality.
For the upper bound we rely on the {Borell-TIS inequality \cite[Thm.\ 2.1.1]{AdT09}}:
{locally abbreviating $\mu:=\mu(H,1)$,}
\begin{align*}
\mu(H,\a) & = \int_0^\infty \p\left(\sup_{t\in[0,1]} B_H(t) > u^{1/\a}\right){\rm d}u \leqslant \int_0^{\mu^{\a}}1\,{\rm d}u
+ \int_{\mu^{\a}}^\infty \exp\left(-\tfrac{1}{2}(u^{1/\a} - \mu)^2\right){\rm d}u \\
& = \int_0^{\mu^{\a}}1\,{\rm d}u + \kk{\a}\int_0^\infty (\mu+y)^{\a-1} \exp(-y^2/2)\,{\rm d}y \\
& \textcolor{black}{\leqslant \mu^\alpha + \max\{1,2^{\a-2}\}\,\a\int_{0}^\infty (\mu^{\a-1} + y^{\a-1}) \exp(-y^2/2)\,{\rm d}y} \\
& \textcolor{black}{= \mu^\alpha + \max\{1,2^{\a-2}\}\,\a\sqrt{\frac{\pi}{2}}\left(\mu^{\alpha-1} + \e|\mathscr N|^{\a-1}\right),}
\end{align*}
\textcolor{black}{where in the third line we used the inequality $(x+y)^p\leq \max\{1, 2^{p-1}\}\cdot(x^p + y^p)$, which holds for any $x,y,p>0$.}
\end{proof}

We further improve the upper bound in \eqref{eq:P2} with the following result. For $H^\circ\in(0,\tfrac{1}{2})$ define
\begin{align}\label{eq:defA}
A(H\,|\, H^\circ) := \frac{2(H-H^\circ)}{1-2H^\circ}.
\end{align}
\begin{lemma}\label{lem:sudakov_supXt}
For any $H\in[H^\circ,\frac{1}{2}]$,
\begin{align*}
\mu(H,1) \leqslant \sqrt{A(H\,|\, H^\circ)}\,\mu(\tfrac{1}{2},1) + \sqrt{1-A(H\,|\, H^\circ)} \,\mu(H^\circ,1),
\end{align*}
where $\mu(\frac{1}{2},1)=\sqrt{\pi/2}$.
\end{lemma}
We remark that $A(H\,|\,H^\circ)\to1$ as $H\uparrow\frac{1}{2}$, so this upper bound is tight at $H=\frac{1}{2}$.
\begin{proof}
Consider a new process $X(\cdot)$, defined by
\begin{equation}\label{eq:defXt}
X(t) := \sqrt{A(H\,|\,H^\circ)}\,B(t) + \sqrt{1-A(H\,|\,H^\circ)}\,B_{{H^\circ}}(t),
\end{equation}
with processes $B(\cdot)\equiv B_{\frac{1}{2}}(\cdot)$ and $B_{{H^\circ}}(\cdot)$ being independent. In the sequel we write for brevity  $A:=A(H\,|\,H^\circ)$. Then,
\begin{align*}
\Var\, X(t) - \Var \,B_H(t) = At + (1-A)t^{2{H^\circ}} - t^{2H} = t(A + (1-A)t^{2{H^\circ}-1}-t^{2H-1}).
\end{align*}
The function {$t\mapsto A + (1-A)t^{2{H^\circ}-1}-t^{2H-1}$} attains its global minimum at $t = 1$, which implies that $\Var \,X(t) \geqslant \Var\, B_H(t)$ for all $t\geqslant 0$.
This means that we are in a position to apply Sudakov's inequality,
see e.g.\ \cite[Thm.~2.6.5]{AdT09}  or  \cite[Prop.~1]{borovkov2017bounds}:
for all $s,t\in [0,1]$ we have
$\e\left[B_H(t)\right]=\e\left[X(t)\right]=0$ and (due to the stationarity of the increments of $B_H(\cdot)$ and $X(\cdot)$)
\begin{align*}
\e\left[ \left( B_H(t)-B_H(s)     \right)^2    \right]
&=
\e\left[ \left( B_H(|t-s|)     \right)^2    \right]\\
&\leqslant
\e\left[ \left( X(|t-s|)     \right)^2    \right]
=
\e\left[ \left( X(t)-X(s)     \right)^2    \right].
\end{align*}
This gives us
\begin{align*}
\mu(H,1)& = \e\left[\sup_{t\in[0,1]} B_H(t)\right]\leqslant \e\left[\sup_{t\in[0,1]} X(t)\right]\\&= \e\left[\sup_{t\in[0,1]}\big\{\sqrt{A}\,B(t) + \sqrt{1-A}\,B_{{H^\circ}}(t)\big\}\right]\\
&\leqslant \sqrt{A}\cdot\e\left[\sup_{t\in[0,1]}B(t)\right] + \sqrt{1-A}\cdot\e\left[\sup_{t\in[0,1]}B_{{H^\circ}}(t)\right],
\end{align*}
which completes the proof.
\end{proof}

Lemma \ref{lem:sudakov_supXt} can be used to improve the upper bound for $\mu(H,1)$ that was presented in Lemma~\ref{lem:a-moment_borel}. The following corollary provides this sharper upper bound.

\begin{corollary}\label{cor:sudakov_Hlow_supXt}
For any $H\in(0,\frac{1}{2}]$, $\mu(H, 1) \leqslant \overline\mu^\circ(H, 1) := \min\{\overline\mu(H,1), \overline\mu'(H,1)\}$, where
\begin{align*}
\overline\mu'(H, 1) := \inf_{H^\circ\in(0,H)}\bigg\{ \sqrt{A(H\,|\, H^\circ)\cdot \pi/2} + \sqrt{1-A(H\,|\, H^\circ)} \, \overline\mu(H,1)\bigg\},
\end{align*}
with $\overline\mu(H,1)$ defined in \eqref{eq:P2}.
\end{corollary}

Notably,  in the neighborhood for $H=\frac{1}{2}$ the upper bound in Corollary \ref{cor:sudakov_Hlow_supXt} improves the upper bound that was found in \cite{borovkov2016new}; see Figure \ref{fig:borovkovVSsudakov}.
More precisely, the figure shows that the above upper bound improves the previous bound $\overline\mu(H,1)$
for $H\in[0.412, 0.5]$. The discontinuities are due to the upper bound in \eqref{eq:P2} being piece-wise constant.

\begin{figure}[h]
        \centering
       \includegraphics[height=0.45\linewidth,width=0.675\linewidth]{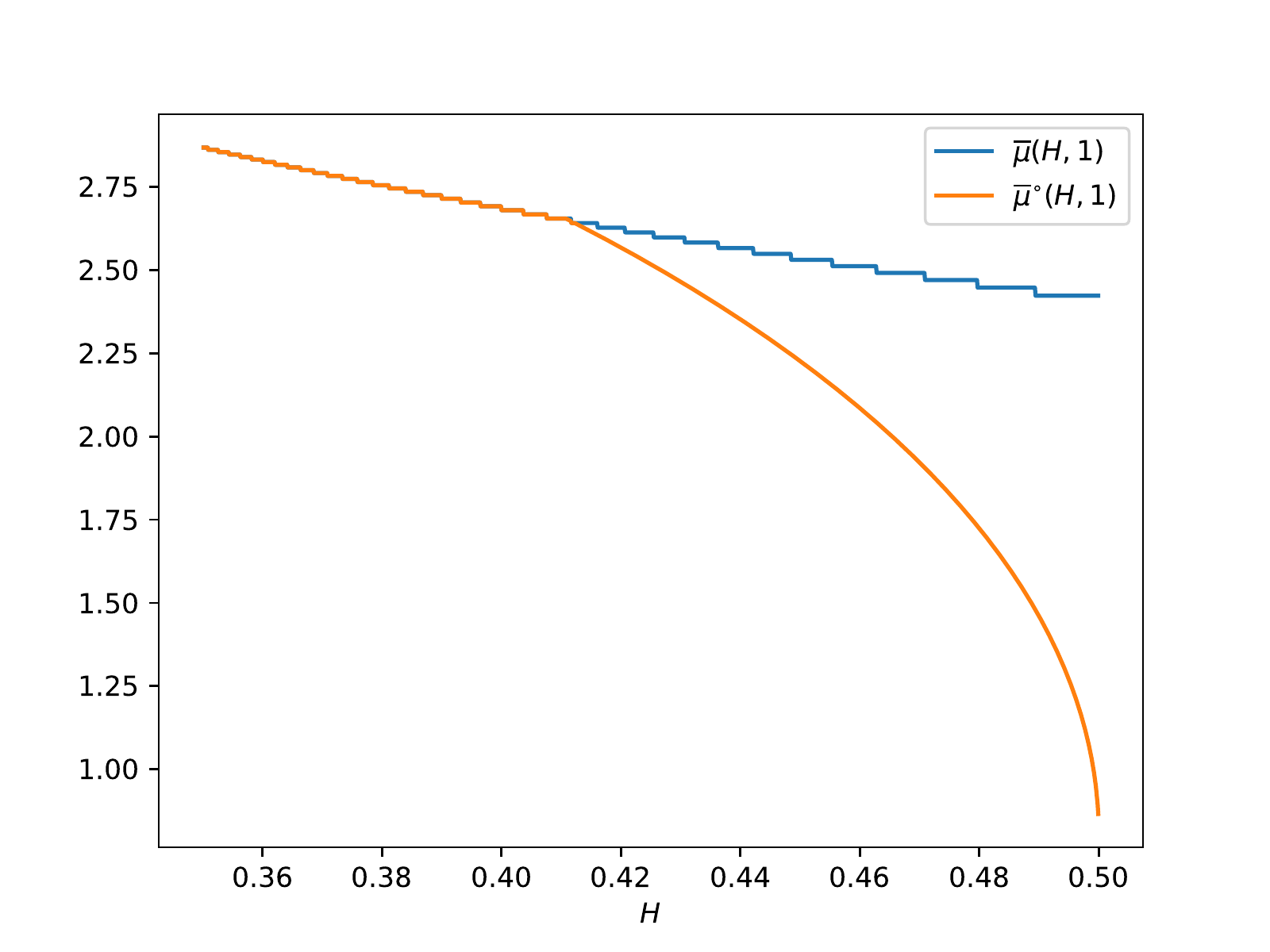}
      \caption{Comparison between the upper bounds for $\mu(H,1)$ given by Lemma \ref{lem:a-moment_borel} and Corollary \ref{cor:sudakov_Hlow_supXt}.
      }\label{fig:borovkovVSsudakov}
\end{figure}

Relying on similar ideas, we can improve the upper bound $\mathscr U_2(\cdot)$ found in Proposition \ref{thm:UB_Cinv}.
Recall that $\mathscr U_2(\cdot)$ has the undesirable property that it has a jump at $H=\frac{1}{2}$, where the exact value of $\mathscr M(H)$ is known ($\mathscr M(\tfrac{1}{2}) = \frac{1}{2}$).
Define, for ${H^\circ}\in(0,\frac{1}{2})$,
\[\gamma(H\,|\,H^\circ) :=\left(\frac{1-2H}{1-2{H^\circ}}\right)^{\frac{1-2{H^\circ}}{2(1-{H^\circ})}}.\]
We remark that $\gamma(H\,|\,H^\circ)\to0$ as  {$H\uparrow\frac{1}{2}$}, so the upper bound in the following lemma is tight at $H=\frac{1}{2}$.

\begin{lemma}\label{lem:sudakov_Hlow}
For any $H\in[{H^\circ},\frac{1}{2}]$,
\begin{align*}
{\mathscr M}(H) \leqslant {\mathscr M}({\tfrac{1}{2}}) +\gamma(H\,|\,H^\circ) {\mathscr M}({{H^\circ}}).
\end{align*}
\end{lemma}

\begin{proof}
Analogously to the proof of Lemma \ref{lem:sudakov_supXt}, the application of Sudakov's inequality, with $A$ defined as in \eqref{eq:defA} and process $X(t)$ defined in \eqref{eq:defXt}, for any fixed $c\in(0,1)$, yields
\begin{align*}
{\mathscr M}(H)& = \e\left[\sup_{t\geqslant 0} \{B_H(t)-t\} \right]\leqslant \e\left[\sup_{t\geqslant 0} \{X(t)-t \}\right]\\&= \e\left[\sup_{t\geqslant 0} \{\sqrt{A}\,B(t) + \sqrt{1-A}\,B_{{H^\circ}}(t) - t \}\right]\\
& = \e\left[\sup_{t\geqslant 0} \left\{\Big(\sqrt{A}\,B(t)-ct\Big) + \Big(\sqrt{1-A}\,B_{{H^\circ}} (t)- (1-c)t\Big)\right\}\right] \\
& \leqslant \e\left[\sup_{t\geqslant 0}\Big\{\sqrt{A}\,B(t)-ct\Big\}\right] + \e\left[\sup_{t\geqslant 0} \Big\{\sqrt{1-A}\,B_{{H^\circ}}(t) - (1-c)t\Big\}\right].
\end{align*}
Finally, an application of the self-smilarity property with $c=A$ yields, after some straightforward computations, the desired upper bound.
\end{proof}

We remark that Lemma \ref{lem:sudakov_Hlow} can be further improved by optimizing over all constants $A$ (or, equivalently, $H^\circ$) and $c$ in the proof; the resulting optimized bound is not explicit (but can be obtained numerically using standard software).

In the following corollary, the upper bound of Lemma \ref{lem:sudakov_Hlow} is combined with Proposition \ref{thm:UB_Cinv}, and in addition optimized over $H^\circ\in(0,H]$.

\begin{corollary}\label{cor:sudakov_Hlow}
For any $H\in(0,\frac{1}{2}]$, ${\mathscr M}(H) \leqslant \mathscr U_2^\circ(H) =  \min\{\mathscr U_2(H), \mathscr U'_2(H)\}$, where
\begin{align*}
\mathscr U_2'(H) := \tfrac{1}{2} + \inf_{H^\circ\in(0,H)} \gamma(H\,|\,H^\circ)\,\mathscr U_2(H^\circ)
\end{align*}
and $\mathscr U_2(\cdot)$ is defined in Proposition $\ref{thm:UB_Cinv}$.
\end{corollary}

\section{Proofs of Theorems \ref{thm:infsup}--\ref{thm:slon}}\label{thm:proofs}
In this section we use the results from the previous sections to establish Theorems \ref{thm:infsup}--\ref{thm:slon}, using the bounds developed in the previous sections.

\begin{proof}[Proof of Theorem \ref{thm:infsup}]
The result concerning $H\uparrow 1$, in \eqref{eq:infsup1}, is an immediate consequence of the results presented in Propositions~\ref{thm:AMI} and \ref{thm:UB_K2001}, in combination with noticing that \[\lim_{H\uparrow 1} H^{\frac{H}{1-H}} = \frac{1}{e}.\]

We continue by showing the result concerning $H\downarrow 0$, i.e., \eqref{eq:infsup0}.
From  {the proofs of} Propositions~\ref{thm:LB_CJB} and  \ref{thm:UB_Cinv}, $\mu(H)$ and $\omega_2(\cdot)$ defined as in Proposition~\ref{thm:UB_Cinv}, we have
\begin{align*}
(1-H)H^{\frac{H}{1-H}} \leqslant \frac{{\mathscr M}(H)}{\mu(H)} \leqslant \omega_2(H) = H^{-H}(1-H)^{-(1-H)},
\end{align*}
which implies that $\lim_{H\downarrow 0} {\mathscr M}(H)/\mu(H) = 1$. From the first part of Lemma \ref{lem:a-moment_borel}, locally using the short notation $\mu:=\mu(H,1)$, and with $\a=({1-H})^{-1}$ \textcolor{black}{and $H<1/2$,}
\begin{align*}
\mu^{\a-1} \leqslant \frac{\mu(H)}{\mu} \leqslant \mu^{\a-1} +
\frac{\textcolor{black}{\a\sqrt{\frac{\pi}{2}}\left(\mu^{\alpha-1} + \e|\mathscr N|^{\a-1}\right)}}{\mu}
\end{align*}
Now, due to the second part of Lemma \ref{lem:a-moment_borel}
we know that $C^- H^{-1/2}\leqslant\mu\leqslant \bar C^+ H^{-1/2}$,
were $C^-=0.2$, and $\bar C^+$ is some constant larger than the $C^+=1.695$. This shows that
\[\textcolor{black}{\frac{\textcolor{black}{\mu^{\alpha-1} + \e|\mathscr N|^{\a-1}}}{\mu}=\frac{1}{\mu^{2-\alpha}} + \frac{\e|\mathscr N|^{\a-1}}{\mu}}
\]
tends to $0$, as $H\downarrow 0$. What is more,
\begin{align*}
(C^-)^{\frac{H}{1-H}} \cdot H^{-\frac{H}{2(1-H)}} \leqslant \mu^{\a-1} \leqslant (C^+)^{\frac{H}{1-H}} \cdot H^{-\frac{H}{2(1-H)}},
\end{align*}
which shows that $\mu^{\a-1}\to1$ as $H\downarrow0$. We thus conclude that $\lim_{H\downarrow 0} {\mathscr M}(H)/\mu = 1$, and hence~\eqref{eq:infsup0} holds.
\end{proof}

\begin{proof}[Proof of Theorem \ref{thm:slon}]
The first part follows directly from the numerical computations underlying Figure \ref{fig:slon}.
The second part follows from Propositions~\ref{thm:AMI} and \ref{thm:UB_K2001} by noticing that, on the interval $H\in[\frac{1}{2},1)$, $H\mapsto H^{{H}/({H-1})}$ monotonically increases from $2$ to $e.$
\end{proof}

\begin{figure}
        \centering
       \includegraphics[height=0.45\linewidth,width=0.675\linewidth]{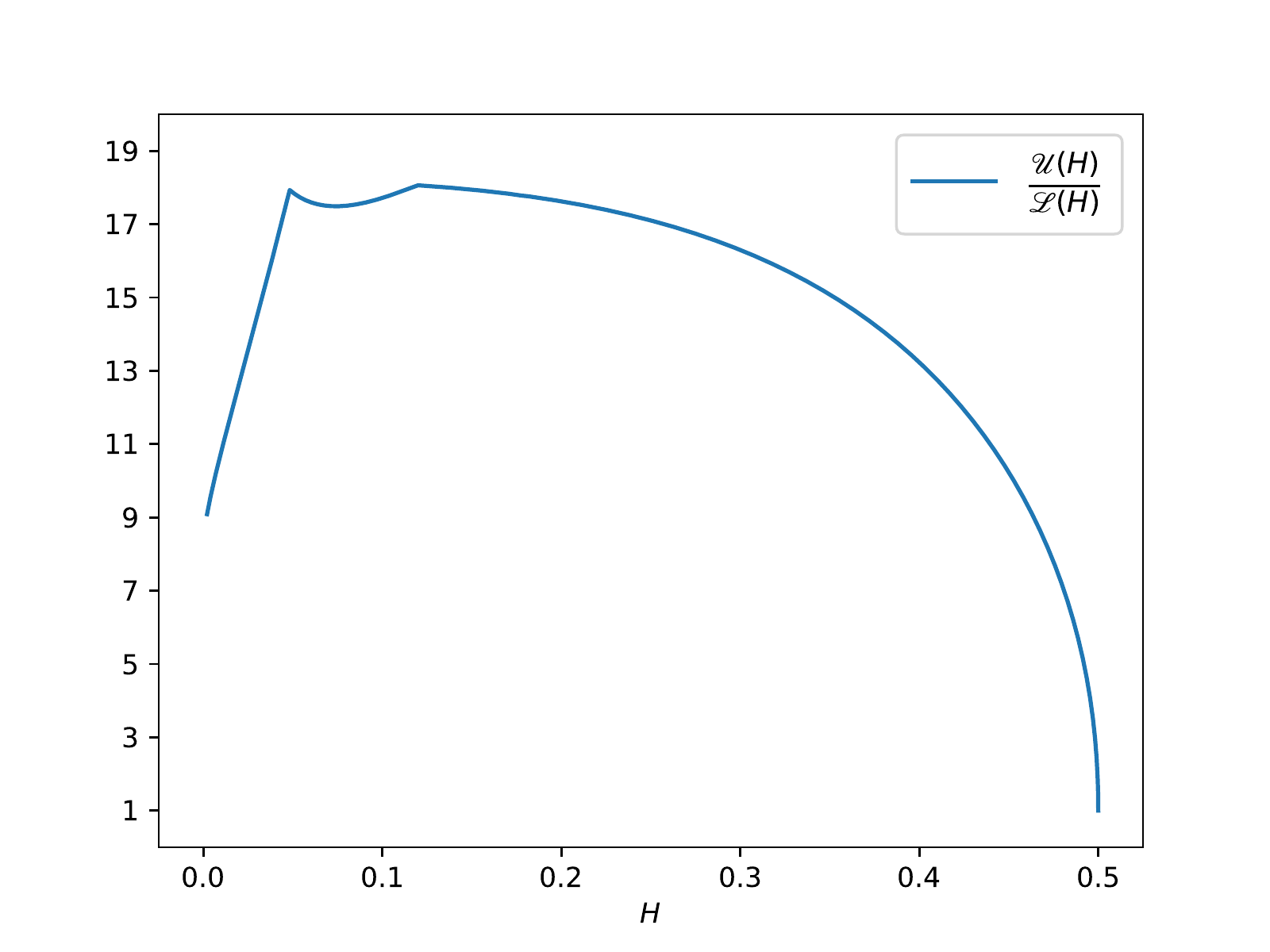}
      \caption{The ratio between the upper and lower bounds for $\mathscr M(H)$ for $H\in[0,\tfrac{1}{2}]$.
      }\label{fig:slon}
\end{figure}

\section{Discussion and conclusions}
\label{CONC} In this paper we have developed upper and lower bounds on the expected supremum of fBm with drift.
Some of these bounds are in closed form, whereas we also include numerical procedures to improve on such bounds.
Future work could aim at further shrinking the gap between the upper and lower bounds.
In addition, one could pursue developing similar bounds for higher moments of the supremum,
or alternatively the variance. Regarding the variance various complications are foreseen,
most notably the  magnitude of the best lower bound on the second moment potentially exceeding the square
of the best upper bound on the first moment, rendering the resulting lower bound on the variance  useless.
Another branch of research could concentrate on finding bounds on the expected supremum for non-fBm Gaussian processes.

\bibliographystyle{plain}

{\small \begin{thebibliography}{111}
\bibitem{ADL}
{\sc R. Adler} (1990). {\it An Introduction to Continuity, Extrema, and Related Topics for General Gaussian Processes}.
Institute of Mathematical Statistics, Hayward, CA, USA.

\bibitem{AdT09}
{\sc R. Adler, J. E. Taylor} (2009). {\it Random fields and geometry}.
Springer Science \& Business Media.


\bibitem{AA}
{\sc S. Asmussen, H. Albrecher} (2010).
{\it Ruin Probabilities}.
World Scientific, Singapore.


\bibitem{BAIL}
{\sc R. Baillie} (1996).
{Long memory processes and fractional integration in econometrics}.
{\it Journal of Econometrics}
{\bf 73}, pp.\ 5-59.

\bibitem{borovkov2016new}
{\sc K. Borovkov, Y. Mishura, A. Novikov, M. Zhitlukhin} (2018).
New bounds for expected maxima of fractional Brownian motion.
{\it Statistics \& Probability Letters} {\bf 137}, pp.\ 142-147.

\bibitem{borovkov2017bounds}
{\sc K. Borovkov, Y. Mishura, A. Novikov, M. Zhitlukhin} (2017).
Bounds for expected maxima of Gaussian processes and their discrete approximations. {\it Stochastics} {\bf 89}, pp.\ 21-37.


\bibitem{CASP}
{\sc A. Caspi, R. Granek, M. Elbaum} (2000).
{Enhanced diffusion in active intracellular transport}.
{\it Physical Review Letters}
{\bf 85}, pp.
5655.


\bibitem{CONT}
{\sc R. Cont} (2005).  {Long-range dependence in financial markets}. In: {\it Fractals in Engineering:
New Trends in Theory and Applications}, {J. L\'evy-V\'ehel, E.  Lutton (eds.)}.  {Springer}, {New York}, NY, USA, pp.\ 159-179.

\bibitem{Deb99}
{\sc K. D\c{e}bicki} (2002). Ruin probability for Gaussian integrated processes.
{\it Stochastic Processes and their Applications} {\bf 98}, pp.\ 151-174.

\bibitem{D1}
{\sc K. D\c{e}bicki} (2001). Asymptotics of the supremum of scaled Brownian motion.
{\it Probability and Mathematical Statistics} {\bf 21}, pp.\ 199-212.


\bibitem{KOS}
{\sc
K. D\c{e}bicki, K. Kosi\'nski, M. Mandjes,  T. Rolski} (2010).
Extremes of multidimensional Gaussian processes. {\it Stochastic
Processes and their Applications} {\bf 120}, pp. 2289-2301.


\bibitem{DM}
{\sc K. D\c{e}bicki, M. Mandjes} (2003).
Exact overflow asymptotics for queues with many Gaussian inputs. {\it Journal of Applied Probability} {\bf 40}, pp. 704-720.

\bibitem{D2}
{\sc K. D\c{e}bicki, Z. Michna, T. Rolski} (2001).  On the supremum from Gaussian processes over infinite horizon.
{\it Probability and Mathematical Statistics} {\bf 18}, pp.\ 83-100.

\bibitem{Die04}
{\sc A. Dieker} (2005). Extremes of {G}aussian processes over an infinite horizon.
{\it Stochastic Processes and their Applications} {\bf 115}, pp. \ 207-248.

\bibitem{DOC}
{\sc N. Duffield, N. O'Connell} (1995). Large deviations and overflow probabilities for general single-server queue, with applications. {\it Mathematical Proceedings of the Cambridge Philosophical Society} {\bf 118}, pp. 363-374.

\bibitem{HUS}
{\sc J. H\"usler, V. Piterbarg} (1999). Extremes of a certain class of Gaussian processes. {\it Stochastic Processes and their Application} {\bf 83}, pp.  257-271.

\bibitem{HuP02}
{\sc J. H\"usler, V. Piterbarg} (2004). On the ruin probability for physical fractional Brownian motion.
{\it Stochastic Processes and their Applications} {\bf 113}, pp.\  315-332.

\bibitem{MAMA}
{\sc A. Malsagov, M. Mandjes} (2019). Approximations for reflected fractional Brownian motion.
{\it Physical Review E}  {\bf  100}, paper no.\ 032120.

\bibitem{MBOOK}
{\sc M. Mandjes} (2007). {\it Large deviations for Gaussian queues.} Wiley,
Chichester, UK.

\bibitem{MMNU}
{\sc M. Mandjes, P. Mannersalo, I. Norros,  M. van Uitert} (2006).
Large deviations of infinite intersections of events in Gaussian
processes. {\it Stochastic Processes and their Applications} {\bf
116}, pp. 1269-1293.

\bibitem{MNG}
{\sc M. Mandjes, I.
Norros,  P. Glynn} (2009). On convergence to stationarity of
fractional Brownian storage. {\it Annals of Applied Probability}
{\bf 18}, pp. 1385-1403.

\bibitem{MU}
{\sc M. Mandjes, M. van Uitert} (2005).
Sample-path large deviations for tandem and priority queues with Gaussian inputs.  {\it Annals of Applied Probability} {\bf  15}, pp. 1193-1226.

\bibitem{MAS}
{\sc L. Massouli\'e, A. Simonian} (1999). Large buffer asymptotics for the queue with FBM input. {\it Journal of Applied Probability} {\bf  36}, pp. 894-906.


\bibitem{MERO}
{\sc Y. Meroz, I. Sokolov} (2015). {A toolbox for determining subdiffusive mechanisms}.
{\it Physics Reports}
{\bf 573}, pp.\  1-30.


\bibitem{METZ}
{\sc R.  Metzler, J.-H. Jeon,  A. Cherstvya, E. Barkaid} (2014). Anomalous diffusion models and their properties: non-stationarity, non-ergodicity, and ageing at the centenary of single particle tracking.
{\it Physical Chemistry Chemical Physics} {\bf 16}, pp.\ 24128-24164.

\bibitem{MONT}
{\sc A. Montanari} (2003).
{Long-range dependence in hydrology}. In: {
\it Theory and applications of long-range dependence},
P. Doukhan, G. Oppenheim, M. Taqqu (eds.).
{Birkh\"auser}, Boston, MA, USA, pp.
{461-472}.

\bibitem{NAR}
{\sc O. Narayan} (1998). Exact asymptotic queue length distribution for fractional Brownian traffic. {\it Advances in Performance Analysis} {\bf 1}, pp. 39-63.

\bibitem{IN}
{\sc I. Norros} (2019). Private communication.

\bibitem{PIT}
{\sc V. Piterbarg} (1996). {\it Asymptotic Methods in the Theory of Gaussian Processes and Fields.} American Mathematical Society, Providence, RI, USA.

\bibitem{REGN}
{\sc B. Regnerand, D. Vucini\'c, C. Domnisoru, T. Bartol, M. Hetzer, D. Tartakovsky, T. Sejnowski} (2013). Anomalous diffusion of single particles in cytoplasm. {\it
Biophysical Journal} {\bf 104}, pp. 1652-1660.

\bibitem{SAGI}
{\sc Y. Sagi, M. Brook, I.  Almog, N. Davidson} (2012).
Observation of anomalous diffusion and fractional self-similarity in one dimension.
{\it Physical Review Letters}
{\bf 108}, paper no.\ 093002.

\bibitem{TAL}
{\sc M. Talagrand} (2014).  {\it Upper and Lower Bounds for Stochastic Processes --- Modern Methods and Classical Problems.} Springer, Heidelberg, Germany.

\bibitem{TAQQ}
{\sc M. Taqqu, W. Willinger, R. Sherman} (1997).
Proof of a fundamental result in self-similar traffic modeling. {\it Computer Communication Review}
{\bf 27}, pp.\ {5-23}.

\bibitem{WIN}
{\sc A. Winkelbauer} (2012). Moments and absolute moments of the normal distribution. arXiv preprint 1209.4340.


\end{thebibliography}}

\end{document}